\documentclass[a4paper, 12pt]{amsart}
\usepackage{amsmath, amsfonts, amssymb, amsthm, graphics, graphicx}
\usepackage{amsaddr}

\usepackage{color}

\usepackage[below]{placeins}
\newtheorem{thm}{Theorem}[section]

\newtheorem{lemma}[thm]{Lemma}
\newtheorem{cor}[thm]{Corollary}
\newtheorem{iremark}[thm]{Remark}
\newtheorem{iexample}[thm]{Example}

\newenvironment{proof*}
{\begin{proof}}
{\renewcommand{\qed}{}\end{proof}}
\newcounter{temp}
\newenvironment{citethm}[1]
{
 \addtocounter{temp}{1}
 \newtheorem*{\alph{temp}theorem}{Theorem #1}
 \begin{\alph{temp}theorem}
}
{\end{\alph{temp}theorem}}
\newenvironment{example}
{ \begin{iexample} \normalfont}
{ \end{iexample}}
\providecommand{\aut}{\mathop{\rm Aut \,}\nolimits}
\providecommand{\sym}{\mathop{\rm Sym \,}\nolimits}

\renewcommand{\\}{\vspace{3mm}}

\title{\bf Rough ends of infinite primitive groups}
\author{\bf Simon M. Smith}
\email{simon.smith@chch.oxon.org, smsmit13@syr.edu}
\address{
   Department of Mathematics, Syracuse University \\
   Syracuse, New York, USA}
\date{\today}
\begin{document}
\maketitle

\begin{abstract}
If $G$ is a group of permutations of a set $\Omega$, then the suborbits of $G$ are the orbits of point-stabilisers $G_\alpha$ acting on $\Omega$. The cardinalities of these suborbits are the subdegrees of $G$. Every infinite primitive permutation group $G$ with finite subdegrees acts faithfully as a group of automorphisms of a locally-finite connected vertex-primitive directed graph $\Gamma$ with vertex set $\Omega$, and there is consequently a natural action of $G$ on the ends of $\Gamma$.

We show that if $G$ is closed in the permutation topology of pointwise convergence, then the structure of $G$ is determined by the length of any orbit of $G$ acting on the ends of $\Gamma$.

Examining the ends of a Cayley graph of a finitely generated group to determine the structure of the group is often fruitful. B. Kr{\"o}n and R. G. M{\"o}ller have recently generalised the Cayley graph to
 what they call a {\it rough Cayley graph}, and they call the ends of this graph the {\it rough ends} of the group.

It transpires that the ends of $\Gamma$ are the rough ends of $G$, and so our result is equivalent to saying that the structure of a closed primitive group $G$ whose subdegrees are all finite is determined by the length of any orbit of $G$ on its rough ends.
\end{abstract}

% ************************************************************************
% ************************** Introduction ********************************
% ************************************************************************
\newpage
\section{Introduction}
\label{section:introduction}

A transitive group of permutations $G$ of a set $\Omega$ is called {\it primitive} if the only $G$-invariant equivalence relations on $\Omega$ are the universal and trivial relations.

Often, to say something useful about an infinite structure, one must first impose some finiteness condition upon it. In this paper we concern ourselves with infinite primitive groups whose point stabilisers have only orbits of finite length in $\Omega$. These orbits are called the {\it suborbits} of a group, and the cardinalities of these orbits are called the {\it subdegrees}. If the subdegrees of a group are all finite, the group is called {\it subdegree finite}. Thus, in this paper we examine infinite subdegree finite groups; such groups have been the subject of much research (see \cite{adeleke&neumann}, \cite{praeger} and \cite{me:SubdegreeGrowth} for example).

Associated with any permutation group $G$ there are directed graphs, called the {\it orbital digraphs} of $G$, which represent $G$ and its action on $\Omega$ in a very natural way. If $\alpha, \beta \in \Omega$ the set $(\alpha, \beta)^G = \{(\alpha^g, \beta^g) \mid g \in G\}$
is called an {\it orbital} of $G$. It is {\it diagonal} if $\alpha = \beta$. An {\it orbital digraph} of $G$ is a digraph whose vertex set is $\Omega$ and whose set of directed edges (or {\it arcs} as they are usually called) is an orbital of $G$.
Orbital digraphs of subdegree finite primitive groups are {\em locally finite}; that is, every vertex is adjacent to only finitely many vertices. Our digraphs are without loops and multiple edges, and unless otherwise stated the terms {\it path} and {\it distance} are used instead of the cumbersome {\it undirected path} and {\it undirected distance}. A digraph $\Gamma$ is {\it connected} if there is a path in $\Gamma$ between any two distinct vertices. A connected component of $\Gamma$ is a maximally connected subdigraph. Because the connected components of a digraph $\Gamma$ induce an equivalence relation on the set of vertices of $\Gamma$, an equivalence relation that is invariant under automorphisms of $\Gamma$, the non-diagonal orbital digraphs of a primitive group $G$ are connected.

For locally finite graphs and digraphs, an {\it end} is an equivalence class of {\it rays} (one-way infinite cycle-free paths), where two rays are equivalent if and only if they have infinitely many vertices in common with a third ray. Intuitively, one may think of them as the ``points of infinity'' of a graph or digraph. An end is {\it thin} if any pairwise disjoint set of rays in it is finite; otherwise it is {\it thick}. It was shown in \cite{me:OrbGraphs} that the ends of any two non-diagonal orbital digraphs of the same subdegree finite primitive group are essentially the same (they are homeomorphic as topological spaces). This is a special case of a deep property of compactly generated totally disconnected locally compact groups, detailed in \cite{kron_moller_analogues_of_Cayley_graphs}.

Our main result, presented in Section~\ref{section:orbits_on_pends}, shows that the structure of an infinite subdegree finite primitive group is determined by the length of any end orbit.

\begin{citethm}{\ref{thm:action_on_ends}} Let $G$ be a subdegree finite primitive group of permutations of an infinite set $\Omega$ that is closed in the natural complete topology of $\sym(\Omega)$. If $\epsilon$ is an end of any connected orbital digraph of $G$ then precisely one of the following holds,
\begin{enumerate}
\item $|\epsilon^G| = 1$ and $G = G_\epsilon$;
\item $|\epsilon^G| = \aleph_0$ and $G \cong G_{\alpha} \ast_{G_{\alpha, \epsilon}} G_\epsilon$;
\item $|\epsilon^G| = 2^{\aleph_0}$ and $G \cong G_{\alpha} \ast_{H_\alpha} H$, for some group $H$ satisfying $G_{\alpha, \epsilon} < H < G$.
\end{enumerate} \end{citethm}

In \cite{kron_moller_analogues_of_Cayley_graphs}  B. Kr{\"o}n and R. G. M{\"o}ller extend the notion of a Cayley graph of a finitely generated group to compactly generated totally disconnected locally compact groups, and they call these graphs {\em rough Cayley graphs}. The ends of any two Cayley graphs of a finitely generated group are the same, and this is also true of rough Cayley graphs. A discussion of their impressive paper is beyond the scope of this short note; it suffices to observe that an infinite primitive permutation group $G$ whose subdegrees are all finite has a rough Cayley graph. In fact, any non-diagonal orbital digraph of $G$ is rough Cayley graph of $G$. Consequently, the ends of any non-diagonal orbital digraph of $G$ are the rough ends of $G$. In Section~\ref{section:rough_ends} we briefly discuss rough Cayley graphs and their ends, and restate our theorem in terms of a primitive group and its action on its rough ends.

% *************************************************************************
% ************************** New Section ********************************
% *************************************************************************
\section{Preliminaries}

It is well-known (see for example \cite[Remark 29.8]{adeleke&neumann}) that if $G$ is a subdegree finite primitive group of permutations of an infinite set $\Omega$, then $\Omega$ is countably infinite. And since $\Omega$ is countable, there is a natural complete topology on the symmetric group $\sym (\Omega)$, that of {\em pointwise convergence}. If we enumerate the set $\Omega$ as $\{\gamma_1, \gamma_2, \gamma_3, \ldots\}$, then a sequence of permutations $(g_n)$ tends to the limit $g$ if and only if, for any $k \geq 1$, we have $\gamma_k^{g_n} = \gamma_k^g$ and $\gamma_k^{g_n^{-1}} = \gamma_k^{g^{-1}}$ for all sufficiently large $n$. A basis for the open sets in this topology consists of all cosets of pointwise stabilisers of finite sets. If $G \leq \sym (\Omega)$, the {\em closure of $G$} is the intersection of all closed subgroups of $\sym (\Omega)$ that contain $G$, and is denoted by $\overline{G}$.

When examining actions on finite subsets of $\Omega$, there is no difference between a group $G$ and its closure $\overline{G}$.

\begin{thm} {\normalfont(\cite[Proposition 2.6]{cameron:oligomorphic})} If $G \leq \sym (\Omega)$ then $G$ and its closure $\overline{G}$ have the same orbits on $n$-tuples of $\Omega$ for all $n \geq 1$. \qed \end{thm}

The following result due to D. M. Evans in \cite{evans:a_note_on_automorphisms} will be fundamental to our argument; here we give the statement of this result as given in \cite{cameron:oligomorphic}.

\begin{thm} {\normalfont(\cite[Theorem 2.8]{cameron:oligomorphic})} \label{thm:evans_oligomorphic} If $G$ and $H$ are closed subgroups of $\sym(\Omega)$, where $\Omega$ is a countable set and $H \leq G$, then either $|G:H| \leq \aleph_0$ or $|G:H| = 2^{\aleph_0}$. The former holds if and only if $H$ contains the pointwise stabiliser in $G$ of some finite subset of $\Omega$. \qed \end{thm}

We will denote the stabiliser of any point $\alpha \in \Omega$ in $G$ by $G_\alpha$, and the setwise and pointwise stabilisers of a subset $\Lambda \subseteq \Omega$ by $G_{\{\Lambda\}}$ and $G_{(\Lambda)}$ respectively. This notation will be extended to ends, with $G_\epsilon$ used to denote the stabiliser in $G$ of the end $\epsilon$.

It was shown in \cite[Theorem 3.11]{me:OrbGraphs} that if $G$ is an infinite subdegree finite primitive group with an orbital digraph with more than one end, then there exists a canonical orbital digraph $\Gamma$ of $G$ that has a very simple, tree-like structure.

Like a tree, every vertex in $\Gamma$ is a {\em cut vertex}; that is, removing any vertex from $\Gamma$ results in a disconnected digraph (such digraphs are said to have {\em connectivity one}). But, unlike a tree,  $\Gamma$ contains non-trivial {\em lobes}: subdigraphs $\Lambda$ that are maximal subject to the condition that no vertex in $\Lambda$ is a cut vertex of $\Lambda$. The lobes of $\Gamma$ are all pairwise isomorphic, have at least three vertices and at most one end. Each lobe of $\Gamma$ is itself a primitive digraph, and any two vertices of $\Gamma$ lie in the same number $m$ of lobes. It was shown in \cite[Theorem 4.5]{me:OrbGraphs} that this canonical orbital digraph of $G$ is essentially unique, and as such can be denoted by $\Gamma(m, \Lambda)$.

Intuitively, one may think of $\Gamma(m, \Lambda)$ as the digraph obtained by glueing together infinitely many copies of $\Lambda$ such that no two copies share more than one vertex, and every vertex of $\Gamma(m, \Lambda)$ lies in $m$ copies of $\Lambda$.

In \cite{jung:watkins} H. Jung and M. Watkins characterise primitive graphs with connectivity one; this was extended to primitive digraphs with connectivity one in \cite{me:InfPrimDigraphs}. These graphs and digraphs all share this tree-like structure.

\begin{thm} {\normalfont(\cite[Theorem 3.3 \& Theorem 3.5]{me:InfPrimDigraphs})}
\label{thm:directed_connectivity_one_digraphs} If $\Gamma$ is a
vertex-transitive digraph with connectivity one, then it is
primitive if and only if the lobes of $\Gamma$ are primitive, pairwise isomorphic digraphs that are not directed cycles of odd prime length, and each lobe has at
least three vertices. \end{thm}

Thus, any locally finite primitive digraph with connectivity one may be written as $\Gamma(m, \Lambda)$, where $m \geq 2$ and $\Lambda$ is some primitive digraph that is not a directed cycle of odd prime length. In particular, note that such graphs have infinitely many ends.

We associate with any connectivity-one digraph $\Gamma$ a bipartite digraph $T$, whose vertices are coloured violet or blue. The violet vertices are the vertices of $\Gamma$, and the blue vertices are the lobes of $\Gamma$. A violet vertex $v$ is adjacent to a blue vertex $b$ in $T$ if and only if $v$ lies in $b$ in $\Gamma$. Adjacent vertices in $T$ are joined by two arcs, one in each direction. This construction yields a tree, which is known as the {\em block-cut-vertex tree} of $\Gamma$.

Applying the Bass-Serre theory of groups acting on trees (\cite{serre:trees}), one obtains the following.

 \begin{thm}{\normalfont(\cite[Theorem 5.2]{me:OrbGraphs})} \label{thm:G_amalgamated_free_product} If $G$ is an subdegree finite primitive group of permutations with an orbital digraph with more than one end, then $G$ has a canonical orbital digraph
$\Gamma(m, \Lambda)$ and
\[G \cong G_{\alpha} \ast_{G_{\alpha, \{\Lambda\}}} G_{\{\Lambda\}},\]
where $\alpha \in V\Lambda$, and $G_{\alpha, \{\Lambda\}}$ is a
non-trivial maximal proper subgroup of $G_{\{\Lambda\}}$ that
fixes no element in $V\Lambda \setminus \{\alpha\}$. \end{thm}

Of course not every group acting transitively on a digraph of the form $\Gamma(m, \Lambda)$ is primitive, and often one needs a quick test for imprimitivity.

\begin{thm} {\normalfont \cite[Theorem 2.5]{me:InfPrimDigraphs}}
\label{thm:Gya_eq_Gyb_not_prim} Let $G$ be a vertex-transitive group
of automorphisms of a connectivity-one digraph $\Gamma$ whose
lobes have at least three vertices, and let $T$ be the
block-cut-vertex tree of $\Gamma$. If there exist distinct
vertices $\alpha, \beta \in V\Gamma$ such that, for some vertex
$x \in (\alpha, \beta)_T$,
\[ G_{\alpha, x} = G_{\beta, x}, \]
then $G$ does not act primitively on $V\Gamma$. \qed \end{thm}

Associated with each orbital $\Delta = (\alpha, \beta)^G$ is its pair $\Delta^* := (\beta, \alpha)^G$, and associated with each suborbit $\Delta(\alpha) := \{\gamma \mid (\alpha, \gamma) \in \Delta\}$ is its pair $\Delta^*(\alpha) := \{\gamma \mid (\alpha, \gamma) \in \Delta^*\}$. In the orbital digraph $(\Omega, \Delta)$, the suborbit $\Delta(\alpha)$ is the set of all vertices $\gamma$ that are adjacent to $\alpha$ with an arc from $\alpha$ to $\gamma$. Similarly, the suborbit $\Delta^*(\alpha)$ is the set of all vertices $\gamma$ that are adjacent to $\alpha$ with an arc from $\gamma$ to $\alpha$. Thus, the out-valency of $\alpha$ in $(\Omega, \Delta)$ is $|\Delta(\alpha)|$, and the in-valency of $\alpha$ in $(\Omega, \Delta)$ is $|\Delta^*(\alpha)|$.

In a locally finite orbital digraph of an infinite primitive group the in-valency and the out-valency of the digraph must be equal. This is a consequence of a lovely theorem by C. E. Praeger.

\begin{thm} {\normalfont(\cite{praeger})} \label{thm:praeger} Let $\Gamma'$ be an infinite connected vertex- and edge-transitive directed graph with finite but unequal in-valency and out-valency. Then there is an epimorphism $\varphi$ from the vertex set of $\Gamma'$ to the set of integers $\mathbb{Z}$ such that $(\alpha, \beta)$ is an edge of $\Gamma'$ only if $\varphi(\beta) = \varphi(\alpha)+1$. \qed \end{thm}

\begin{cor} \label{cor:paired_suborbits_have_same_size} Suppose $G'$ is a primitive group of permutations of an infinite set $\Omega$, and every suborbit of $G'$ is finite. If $\alpha \in \Omega$ and $\Delta(\alpha)$ and $\Delta^*(\alpha)$ are paired $\alpha$-suborbits then $|\Delta(\alpha)| = |\Delta^*(\alpha)|$. \qed \end{cor}

% *************************************************************************
% ************************** New Section ********************************
% *************************************************************************
\section{Orbits on ends}
\label{section:orbits_on_pends}

Suppose $G$ is an infinite primitive subdegree finite group of permutations with canonical orbital digraph $\Gamma(m, \Lambda)$, and let $T$ be the block-cut-vertex tree of $\Gamma$. Finally, assume that $G$ is closed, and recall that a group and its closure have the same orbits on $n$-tuples of $\Omega$, for all $n\geq 1$.

Using the following three lemmas, it is possible to scrutinise the action of $G$ on the thin ends of any connected orbital digraph of $G$.

\begin{lemma} \label{lemma:Ga_leq_closure Gepsilon_then_Ga_leq_Gepsilon} For any thin end $\epsilon$ of $\Gamma$, if $\underline{\alpha}$ is an $n$-tuple of vertices with $G_{\underline{\alpha}} \leq \overline{G_\epsilon}$, then $G_{\underline{\alpha}} \leq G_\epsilon$. \end{lemma}

\begin{proof} Enumerate the set $V\Gamma$ as $\{\gamma_1, \gamma_2, \gamma_3, \ldots \}$, and recall that an element $g \in G$ is the limit of a sequence of permutations $(g_n)$ of $V\Gamma$ if and only if, given $k \geq 1$, there exists an integer $N_k$ such that, for all $n \geq N_k$ we have $\gamma_i^{g_n} = \gamma_i^g$ and $\gamma_i^{g_n^{-1}} = \gamma_i^{g^{-1}}$ whenever $1 \leq i \leq k$.

Suppose $G_{\underline{\alpha}} \leq \overline{G_\epsilon}$ but $G_{\underline{\alpha}} \not \leq G_\epsilon$, and fix
$g \in G_{\underline{\alpha}} \setminus G_\epsilon.$
Since $g \in \overline{G_\epsilon}$, there exists a sequence $(g_n)$ of permutations in $G_\epsilon$ such that
$g_n \rightarrow g.$
If $\underline{\alpha} = (\alpha_1, \ldots, \alpha_n)$ then choose $k \geq 1$ such that $\{\alpha_1, \ldots, \alpha_n\} \subseteq \{\gamma_1, \ldots, \gamma_k\}$. Since $g_n \rightarrow g$, there exists an integer $N_k$ such that, for all $n \geq N_k$ we have $\gamma_i^{g_n} = \gamma_i^g$ whenever $1 \leq i \leq k$. Hence, for all $n \geq N_k$,
\[\underline{\alpha}^{g_n} = \underline{\alpha}.\]

Recall $T$ is the block-cut-vertex tree of $\Gamma$, and let $[\alpha_1, \epsilon)_T$ be the unique half-line of $T$ rooted at $\alpha_1$ lying in the end $\epsilon$. For all $n \geq N_k$ we have $g_n \in G_{\alpha_1, \epsilon}$, therefore $g_n$ fixes every vertex on $[\alpha_1, \epsilon)_T$. Since $g$ lies in $G_{\alpha_1}$ but not $G_{\epsilon}$, there exists $x \in [\alpha_1, \epsilon)_T$ such that $x^g \not = x$; furthermore, if $y \in [x, \epsilon)_T$ then $y^g \not = y$. Therefore, we may choose $\beta \in [x, \epsilon)_T \cap V\Gamma$ such that $\beta^g \not = \beta$. Since $\beta \in V\Gamma$, there is an integer $k^\prime$ such that $\beta = \gamma_{k^\prime}$. And because $g_n \rightarrow g$, there exists $N_{k^\prime} \geq 1$ such that for all $n \geq N_{k^\prime}$ we have $\gamma_{k^\prime}^{g_n} = \gamma_{k^\prime}^g$. Fix $N:= \max (N_k, N_{k^\prime})$. For all $n \geq N$, we have $g_n \in G_{\underline{\alpha}}$ but $\beta^{g_n} = \beta^g \not = \beta$, so $g_n \not \in G_\epsilon$, a contradiction. \end{proof}

\begin{lemma} \label{lemma:if_Ga_leq_Gepsilon} If there is an end $\epsilon$ of $\Gamma$ and an $n$-tuple $\underline{\alpha} = (\alpha_1, \ldots, \alpha_n)$ of vertices of $\Gamma$ such that $G_{\underline{\alpha}} \leq G_\epsilon$, then, for $1 \leq i \leq n$, the orbit $\epsilon^{G_{\alpha_i}}$ is finite. \end{lemma}

\begin{proof} Suppose $G_{\underline{\alpha}} \leq G_\epsilon$. Then, for all $i$ satisfying $1 \leq i \leq n$ we have $G_{\underline{\alpha}} \leq G_{\alpha_i, \epsilon}$. Therefore
\[|G_{\alpha_i} : G_{\underline{\alpha}}|
    = |G_{\alpha_i} : G_{\alpha_i, \epsilon}| |G_{\alpha_i, \epsilon} :
        G_{\underline{\alpha}}|
    = |\epsilon^{G_{\alpha_i}}| |G_{\alpha_i, \epsilon} :
        G_{\underline{\alpha}}|.\]
Observe
\[|G_{\alpha_1} : G_{\underline{\alpha}}|
    =  \prod_{m=1}^{n-1} |G_{\alpha_{1} \ldots \alpha_{m}} :
        G_{\alpha_{1} \ldots \alpha_{m+1}}|
    =  \prod_{m=1}^{n-1} |\alpha_{m+1}^{G_{\alpha_{1} \ldots
        \alpha_{m}}}|.\]

Since $G$ is subdegree finite, each orbit $|\alpha_{m+1}^{G_{\alpha_{1} \ldots \alpha_{m}}}|$ is finite, so $|G_{\alpha_1} : G_{\underline{\alpha}}|$ is finite. A similar argument shows $|G_{\alpha_i}:G_{\underline{\alpha}}|$ is finite for each $i$ satisfying $1 \leq i \leq n$. We have already seen that
\[|G_{\alpha_i} : G_{\underline{\alpha}}| = |\epsilon^{G_{\alpha_i}}| |G_{\alpha_i, \epsilon} : G_{\underline{\alpha}}|,\]
so one may deduce that $|\epsilon^{G_{\alpha_i}}|$ is finite. \end{proof}

\begin{lemma} \label{lemma:thin_end_in_small_Ga_orbit_for_all_a} If $\epsilon$ is a thin end of $\Gamma$ with $|\epsilon^G| < 2^{\aleph_0}$ then, for all $\alpha \in V\Gamma$, the orbit $\epsilon^{G_{\alpha}}$ is finite. \end{lemma}

\begin{proof} Let $\epsilon$ be a thin end of $\Gamma$, and let $H := \overline{G_\epsilon}$. Now
\[|\epsilon^G| = \ |G:G_\epsilon|
= \ |G:H| |H:G_\epsilon|.\]

If $|\epsilon^G| < 2^{\aleph_0}$, then $|G:H|<2^{\aleph_0}$, and, since $G$ and $H$ are both closed, we may apply Theorem~\ref{thm:evans_oligomorphic} to deduce that there exists an $n$-tuple of vertices $\underline{\alpha}$ such that $G_{\underline{\alpha}} \leq H$, and $|G:H| \leq \aleph_0$.

Suppose $|\epsilon^G| < 2^{\aleph_0}$ and $\alpha = (\alpha_1, \ldots, \alpha_n)$ is an $n$-tuple of vertices of $\Gamma$ such that $G_{\underline{\alpha}} \leq \overline{G_\epsilon}$. By Lemma~\ref{lemma:Ga_leq_closure Gepsilon_then_Ga_leq_Gepsilon} we have $G_{\underline{\alpha}} \leq G_\epsilon.$
Fix $\alpha \in V\Gamma$ and let $\underline{\alpha}^\prime:=(\alpha, \alpha_1, \ldots, \alpha_n)$; then $G_{\underline{\alpha}^\prime} \leq G_{\underline{\alpha}} \leq G_\epsilon$, and so, from Lemma~\ref{lemma:if_Ga_leq_Gepsilon}, the orbit $\epsilon^{G_\alpha}$ is finite. \end{proof}

We now prove our main result.

\begin{thm} \label{thm:action_on_ends}  Let $G$ be a subdegree finite primitive group of permutations of an infinite set $\Omega$ that is closed in the natural complete topology of $\sym(\Omega)$. If $\epsilon$ is an end of any connected orbital digraph of $G$ then precisely one of the following holds,
\begin{enumerate}
\item $|\epsilon^G| = 1$ and $G = G_\epsilon$;
\item $|\epsilon^G| = \aleph_0$ and $G \cong G_{\alpha} \ast_{G_{\alpha, \epsilon}} G_\epsilon$;
\item $|\epsilon^G| = 2^{\aleph_0}$ and $G \cong G_{\alpha} \ast_{H_\alpha} H$, for some group $H$ satisfying $G_{\alpha, \epsilon} < H < G$.
\end{enumerate} \end{thm}

\begin{proof} If an orbital digraph of $G$ has precisely one end, then all non-diagonal orbital digraphs of $G$ have just one end and (i) holds. So, suppose $G$ has an orbital digraph with more than one end. By Theorem~\ref{thm:G_amalgamated_free_product}, $G$ has a canonical orbital graph $\Gamma = \Gamma(m, \Lambda)$, and we may write
$G \cong G_{\alpha} \ast_{G_{\alpha, \{\Lambda\}}} G_{\{\Lambda\}}$. Furthermore, since the ends of any two connected orbital digraphs of $G$ are the same, we may assume $\epsilon$ is an end of $\Gamma$.

If $\epsilon$ is a thick end of $\Gamma$, then $\epsilon$ must be the end of some lobe $\Lambda'$ of $\Gamma$. Since $G$ permutes the lobes of $\Gamma$ transitively (because $G$ acts transitively on the arc set of $\Gamma$), the subgroups $G_{\{\Lambda\}}$ and $G_{\{\Lambda'\}}$ are isomorphic; every lobe contains precisely one end (which is thick); and $G_\epsilon \cong G_{\{\Lambda\}}$. There are countably many lobes in $\Gamma$, so $|\epsilon^G| = \aleph_0$ and (ii) holds.

Finally, suppose $\epsilon$ is a thin end of $\Gamma$, and $|\epsilon^G| < 2^{\aleph_0}$. Let $T$ be the block-cut-vertex tree of $\Gamma$. Since $\epsilon$ is a thin end of $\Gamma$, it is an end of $T$. By Lemma~\ref{lemma:thin_end_in_small_Ga_orbit_for_all_a}, if $|\epsilon^G| < 2^{\aleph_0}$ then for all $\alpha \in V\Gamma$, the orbit $\epsilon^{G_{\alpha}}$ is finite. Fix $\alpha \in V\Gamma$. Since $\epsilon^{G_{\alpha}}$ is finite, there exists a vertex $\alpha' \in V\Gamma$ that lies in $[\alpha, \epsilon)_T$ such that $G_{\alpha, \alpha'} \leq G_\epsilon$. Put $m:=|\alpha^{\prime \ G_\alpha}|$, and note $m = |\epsilon^{G_\alpha}|$. By Corollary~\ref{cor:paired_suborbits_have_same_size}, $|\alpha^{G_{\alpha'}}| = m$.

Now choose $\beta \in (\alpha', \epsilon)_T$. A similar argument shows there exists $\beta' \in (\beta, \epsilon)_T$ such that $G_{\beta, \beta'} \leq G_\epsilon$.

By choosing $\beta$ and $\beta'$ in this way, we ensure that $G_{\alpha, \beta'} \leq G_{\beta, \beta'}$ and $|\beta^{\prime G_\alpha}| = |\epsilon^{G_\alpha}| = m$. We again apply Corollary~\ref{cor:paired_suborbits_have_same_size} to deduce $|\alpha^{G_{\beta'}}| = m$. Hence
\begin{align*} m    =& \ |\alpha^{G_{\beta'}}|\\
                    =& \ |G_{\beta'}:G_{\alpha, \beta'}|\\
                    =& \ |G_{\beta'}:G_{\beta, \beta'}| |G_{\beta, \beta'}:G_{\beta' \alpha}|. \end{align*}

If $|G_{\beta, \beta'}:G_{\beta' \alpha}| = 1$ then $G_{\beta, \beta'} \leq G_\alpha$ and so $G_{\alpha, \alpha'} \leq G_{\beta, \beta'}$ and $G_{\beta, \beta'} \leq G_{\alpha, \alpha'}$, with $\alpha', \beta \in (\alpha, \beta')_T$ and $[\alpha, \alpha']_T \cap (\beta, \beta']_T = \emptyset$. Applying Theorem~\ref{thm:Gya_eq_Gyb_not_prim} we see $G$ is not primitive, which contradicts our original assumption; thus, we must have
$|G_{\beta, \beta'}:G_{\beta' \alpha}| > 1$. This implies $|G_{\beta'}:G_{\beta, \beta'}| < m$; that is, $|\beta^{G_{\beta'}}| < m$. However, $|\beta^{G_{\beta'}}| = |\beta^{\prime G_{\beta}}|$ by Corollary~\ref{cor:paired_suborbits_have_same_size}; furthermore, $\beta$ and $\beta'$ we chosen so that $G_{\beta, \beta'} \leq G_\epsilon$. Hence $|\epsilon^{G_\beta}| < m$.

Since $\alpha$ was chosen arbitrarily, we may now set $\alpha:=\beta$ and repeat the above argument. Eventually, we find a vertex $\alpha$ such that $|\epsilon^{G_\alpha}| = 1$. However, this is also a contradiction. Indeed, let $x$ be the vertex of $T$ that is adjacent to $\alpha$ in $[\alpha, \epsilon)_T$. This vertex corresponds to some lobe $\Lambda$ of $\Gamma$ that contains $\alpha$. By Theorem~\ref{thm:G_amalgamated_free_product}, the group $G_{\alpha, \{\Lambda\}}$ fixes no vertex in $V\Lambda \setminus \{\alpha\}$, and therefore cannot fix the half-line $[\alpha, \epsilon)_T$. Whence, we must have $|\epsilon^{G}| = 2^{\aleph_0}$.

Furthermore, since $x$ is the vertex of $T$ adjacent to $\alpha$ in $[\alpha, \epsilon)_T$,
\[G_{\alpha, \epsilon} \leq G_x.\]
The vertex $x$ corresponds to some lobe of $\Gamma$, and so $G_x \cong G_{\{\Lambda\}}$. Finally, we note that $G_{\alpha, \epsilon} \not = G_x$ because $|\epsilon^{G}| = 2^{\aleph_0}$. Taking $H = G_x$ we have $G \cong G_{\alpha} \ast_{H_\alpha} H$, with $G_{\alpha, \epsilon} < H < G$ by Theorem~\ref{thm:G_amalgamated_free_product}.
\end{proof}

As the following three examples show, this theorem cannot be improved upon. Indeed, the result fails if one does not require $G$ to be closed; it also fails if one does not require $G$ to be primitive. Furthermore, there are infinite imprimitive subdegree finite closed permutation groups that satisfy the conclusions of the theorem.

\begin{example}
Given any group $G$ of permutations of a countable set $\Omega$, one can find a subgroup $H$ of $G$ that is countable with the same orbits on tuples of $\Omega$ as $G$, using a method described in \cite{cameron:oligomorphic}.

Since $\Omega$ is countable, the set of all tuples of $\Omega$ is a countable union of countable sets, and is therefore countable. Enumerate pairs $(\underline{\alpha_n}, \underline{\beta_n})$ of tuples for which $\underline{\alpha_n}$ and $\underline{\beta_n}$ lie in the same orbit of $G$. For each natural number $n$ choose $g_n \in G$ such that $\underline{\alpha_n}^{g_n} = \underline{\beta_n}$, and let $H$ be the subgroup of $G$ generated by the elements $g_n$. This group is countable, and has the same orbits on tuples as $G$.

If $G$ is taken to be a closed arc-transitive primitive group of automorphisms of a locally finite orbital digraph $\Gamma$ with more than one end, and $H$ is constructed as described, then the orbit of $G$ on any thin end $\epsilon$ of $\Gamma$ with have length $2^{\aleph_0}$ by Theorem~\ref{thm:action_on_ends}, while the length of the orbit of $H$ on $\epsilon$ will be at most $\aleph_0$. Furthermore, because $G$ and $H$ have the same orbits on tuples, their orbital digraphs must be the same. In particular, all non-diagonal orbital digraphs of $H$ must be connected, and so $H$ is primitive by D. G. Higman's famous result (\cite{higman}). The group $H$ fails to satisfy the conclusions of Theorem~\ref{thm:action_on_ends} because $H$ is not closed.
\end{example}

\begin{example}
Let $S_2$ be the symmetric group on two letters, and let $S_3$ be the symmetric group on three letters, and let $G = S_2 \ast S_3$ be the free product of $S_2$ and $S_3$. This group acts without inversion on a bivalent tree $T$ in which one part of the natural bipartition of $T$ has valency two and the other has valency three, and $G$ acts transitively on each part of the bipartition. Colour the vertices with valency two violet, and the vertices with valency three blue. Let $V$ be the set of violet vertices of $T$, and let $\Gamma$ be the orbital digraph $(V, (\alpha, \beta)^G)$, where $\alpha, \beta \in V$ are at distance two in $T$.

If $x \in T$ is the vertex lying between $\alpha$ and $\beta$, then $G_x$ acts like $S_3$ on the vertices in $T$ adjacent to $x$. Whence, $\Gamma$ has connectivity one, and the lobes of $\Gamma$ are isomorphic to $K_3$, the complete digraph on three vertices, with each vertex in $\Gamma$ lying in precisely two lobes. Thus, $\Gamma = \Gamma(2, K_3)$, with block-cut-vertex tree $T$.

Now $\Gamma(2, K_3)$ is an orbital digraph of $G$ acting on $V$. The stabiliser $G_\alpha$ is isomorphic to $S_3$ and is therefore finite. Whence $G_\alpha$ is closed in the permutation topology, so we now replace $G$ with its closure and assume henceforth that $G$ is closed.

Again since $G_\alpha$ is finite, by Theorem~\ref{thm:Gya_eq_Gyb_not_prim}, the group $G$ does not act primitively on $\Gamma(2, K_3)$, although the full automorphism group of $\Gamma(2, K_3)$ is primitive by Theorem~\ref{thm:directed_connectivity_one_digraphs}. If $\epsilon$ is any thin end of $\Gamma$, then the end orbit $\epsilon^G$ has at most countably infinite length.

The closed group $G$ fails to satisfy the conclusions of Theorem~\ref{thm:action_on_ends} because $G$ is not primitive.
\end{example}

\begin{example}
Let $C_5$ denote a directed cycle on  five vertices, and let $\Gamma$ be the digraph $\Gamma(3, C_5)$. The group $\aut \Gamma$ acts imprimitively on $V\Gamma$ by Theorem~\ref{thm:directed_connectivity_one_digraphs}, but satisfies the conclusions of Theorem~\ref{thm:action_on_ends} (in fact the group $\aut \Gamma$ is of type (iii)).
\end{example}

% *************************************************************************
% ************************** New Section ********************************
% *************************************************************************
\section{rough ends}
\label{section:rough_ends}

In this section we relate Theorem~\ref{thm:action_on_ends} to the exciting work of B. Kr{\"o}n and R. G. M{\"o}ller (\cite{kron_moller_analogues_of_Cayley_graphs}).

Recall that any permutation group $G$ of a countable set can be considered a topological group by imposing upon it the permutation topology. A basis for the open sets of this topology are the cosets of pointwise stabilisers of finite sets. It is a simple exercise to check that if $G$ is closed then the pointwise stabiliser of a single element in $\Omega$ is also closed.

If $G$ is a topological group, then a connected graph $\Gamma$ is said to be a {\em rough Cayley graph} of $G$ if $G$ acts transitively on $\Gamma$ and the stabilisers of vertices are compact open subgroups of $G$.

If $G$ has a rough Cayley graph then the ends of any two rough Cayley graphs are homeomorphic and can therefore be considered to be the same (\cite[Theorem 2.7 and Section 3.1.1 ]{kron_moller_analogues_of_Cayley_graphs}). The {\em rough ends} of $G$ are defined to be the ends of a rough Cayley graph of $G$ (\cite[Definition 3.12]{kron_moller_analogues_of_Cayley_graphs}).

\begin{thm} {\normalfont(\cite[Corollary 1]{moller:FC})}
If $G$ acts as a group of automorphisms of a locally finite connected graph $\Gamma$ such that $G$ is transitive on the vertex set of $\Gamma$ and the stabilisers of the vertices in $\Gamma$ are compact and open, then $G$ is compactly generated.
\end{thm}

\begin{lemma}  {\normalfont(\cite[Lemma 2]{{woess:topological_groups_and_infinite_graphs}})}
If $\Gamma$ is a vertex transitive infinite connected locally finite graph then the closure (in the permutation topology) of any subset $U \subseteq \aut(\Gamma)$ is compact if and only if the orbit $\alpha^U$ is finite for all $\alpha \in V\Gamma$.
\end{lemma}

Thus, if $G$ is a closed primitive group of permutations of a countable set, and the subdegrees of $G$ are all finite, then $G$ acts transitively on a non-diagonal orbital digraph $\Gamma$, which is necessarily connected, and the stabilisers of vertices in $\Gamma$ are open and compact. Whence $\Gamma$ is a rough Cayley graph of $G$, and the rough ends of $G$ are the ends of $\Gamma$.

We conclude by restating our main result as a theorem about a primitive group and its rough ends.

\begin{thm} \label{thm:action_on_ends}  Let $G$ be a subdegree finite primitive group of permutations of an infinite set $\Omega$ that is closed in the natural complete topology of $\sym(\Omega)$. If $\epsilon$ is a rough end of $G$ then precisely one of the following holds.
\begin{enumerate}
\item $|\epsilon^G| = 1$ and $G = G_\epsilon$;
\item $|\epsilon^G| = \aleph_0$ and $G \cong G_{\alpha} \ast_{G_{\alpha, \epsilon}} G_\epsilon$;
\item $|\epsilon^G| = 2^{\aleph_0}$ and $G \cong G_{\alpha} \ast_{H_\alpha} H$, for some group $H$ satisfying $G_{\alpha, \epsilon} < H \leq G$.
\end{enumerate} \end{thm}

% *******************
% *	     	    *
% * Acknowledgements *
% *  	     	    *
% *******************

%\vspace{1cm}
\\

{\em \noindent Acknowledgements.} Some of the results in this paper are taken from the author's DPhil thesis, completed
at the University of Oxford, under the supervision of Peter Neumann and funded by the EPSRC; the remainder was completed at Syracuse University as a Philip T Church Postdoctoral Fellow. The author would
like to thank Dr Neumann and the EPSRC.

% *************** BIBLIOGRAPHY ******************

\end{document}